\documentclass[a4paper,twoside,draft,12pt]{amsart}
\usepackage[utf8]{inputenc}
\usepackage[T1]{fontenc}
\usepackage[english]{babel}
\usepackage{tgpagella}
\usepackage{hyperref}
\usepackage{amscd} 
\usepackage{csquotes}
\usepackage{fancyhdr}
\usepackage[minbibnames=1, maxbibnames=99, backend=bibtex,style=numeric]{biblatex}
\usepackage{amsthm, amsfonts, amssymb, amsmath, latexsym, enumerate,array}
\usepackage{amscd} 
\usepackage[all]{xy}
\usepackage{pstricks,graphicx}
\usepackage{stmaryrd}
\usepackage{hyperref,mdwlist,mathrsfs}
\usepackage{tikz}
\usepackage{pgf,tikz}
\usetikzlibrary{arrows}
\usepackage{dirtytalk}

\usepackage{color}

\usepackage{amssymb}
\usepackage{amsmath}
\usepackage{amsthm}
\usepackage{paralist}

\newcommand{\kk}{\ensuremath{K}}

\newcommand{\Cat}{\ensuremath{\mathrm{Cat}}}
\newcommand{\rk}{\ensuremath{\mathrm{rk}}}

\newtheorem{theorem}{Theorem}[section]

\newtheorem*{proposition*}{Proposition}
\theoremstyle{definition}
\newtheorem{definition}[theorem]{Definition}
\newtheorem{example}[theorem]{Example}
\newtheorem{remark}[theorem]{Remark}

\DeclareMathOperator{\Ann}{Ann}

\title{}

\author[]{Nasrin Altafi}
\address{Department of Mathematics and Statistics, Queen’s University, Kingston, Ontario, Canada and Department of Mathematics, KTH Royal Institute of Technology, Sweden}
\email{nasrinar@kth.se}

\thanks{\hspace{-15pt}The author was supported by the grant VR2021-00472. The author would like to thank the organizers of the AMS-EMS-SMF meeting at Grenoble 18-22 July 2022, special session on \say{Deformation of Artinian algebras and Jordan type}. The computer software Macaulay2 \cite{M2} was important for the computations in this paper.}

\theoremstyle{definition}
\newtheorem{problem}{Problem}

\begin{document}
\title{A  note on Jordan types and Jordan degree types}
\maketitle

\begin{abstract}
We discuss whether the Jordan degree type encodes \break more information about graded artinian Gorenstein algebras than the Jordan type for linear forms. We show that in codimension two, the Jordan type determines the Jordan degree type. We provide examples showing that this is no longer the case in higher codimensions. 
\end{abstract}

\section{Introduction}

For a standard graded artinian algebra $A=\kk[x_1,\dots ,x_n]/I$ and a linear form $\ell\in A_1$ the \textit{Jordan type} $P_{\ell,A}$ is the partition of $\dim_\kk A$ determining the Jordan block decomposition of the multiplication map by $\ell$ on $A$.  Jordan type determines the weak and strong Lefschetz properties of artinian algebras. A graded artinian algebra A is said to satisfy the weak Lefschetz property (WLP) if multiplication map by a linear form on A has maximal rank in every degree. If this holds for all powers of a linear form the algebra A is said to have the strong Lefschetz property (SLP). It is known that an artinian algebra $A$ has the WLP if there is a linear form $\ell$ where the number of parts in $P_{\ell,A}$ is equal to the Sperner number of $A$, the maximum value of the Hilbert function $h_A$. Also A has the SLP if there is a linear form $\ell$ such that $P_{\ell,A} = h_A^\vee$ the conjugate partition of $h_A$, see \cite{Lefbook} and \cite{IMM} for the refined version. Jordan type of a linear form for an artinian algebra captures more information than the weak and Strong Lefschetz properties. For recent papers on Jordan types see \cite{IKVZ, IMM, IMM2} and their references. 

The study of the Lefschetz properties of artinian Gorenstein algebras has recently become one of the most attractive topics in commutative algebra and algebraic geometry. Artinian Gorenstein (AG) algebras are well structured and have many connections to other areas, such as combinatorics and algebraic geometry. More recently, the study of Jordan types of such algebras has become very attractive. Costa and Gondim in \cite{CG} study Jordan types of AG algebras for a general linear form by computing the ranks of associated Hessian and mixed Hessian matrices introduced in \cite{MW} and \cite{GZ}. 

Iarrobino, Macias Marques, and McDaniel in \cite{IMM} study Jordan \break types of graded and local artinian algebras. They introduce the \emph{Jordan degree type} that determines the Jordan type as well as the initial degrees associated to different parts. They show that the Jordan degree type is not equal to the Jordan type by providing two artinian algebras with the Hilbert function $(1,3,5,4,2,1)$ for which they share a Jordan type, although their Jordan degree types are different. 
In this paper, we discuss whether or not there is one Jordan degree types for each Jordan type for AG algebras. In Theorem \ref{theorem}, we show that, for AG algebras of codimension two, the Jordan type is equivalent to the Jordan degree type. We then show that the Jordan degree type is a finer invariant than the Jordan type for AG algebras of codimension at least three. In Example \ref{ex} we provide two AG algebras with the Hilbert function $(1,3,6,9,9,9,6,3,1)$ for which they have equal Jordan type for a linear form but different Jordan degree types. It has been computationally verified that this example has the smallest socle degree among AG algebras of codimension three. We also give such example in codimension four for AG algebras with Hilbert function $(1,4,7,7,4,1)$, see Example \ref{codim4Ex}.  



\section{Jordan types and Jordan degree types of AG algebras}
Let $S = \kk[x_1,\dots ,x_n]$ be a polynomial  ring equipped with standard grading over a field $\kk$ of characteristic zero. Let $A=S/I$ be a graded artinian  ( its Krull dimension is zero) algebra where $I$ is an homogeneous ideal. 
The \emph{Hilbert function} of a graded artinian  algebra $A=S/I$ is  a vector of non-negative integers and we denote it by $h_A=(1,h_1,\dots ,h_d)$ where  $h_A(i)=h_i=\dim_{\kk}(A_i)$. The integer $d$ is called the \emph{socle degree} of $A$, that is the largest integer $i$ such that $h_A(i)>0$.
Let $R=\kk[X_1,\dots , X_n ]$ be the Macaulay dual ring of $S$. Given a homogeneous ideal $I\subset S$ the \emph{inverse system} of $I$ is defined to be a graded $S$-module $M\subset R$ such that $S$ acts on  $R$ by differentiation. For more details of Macaulay's inverse system see \cite{Geramita} and \cite{IK}. From  a  result by Macaulay \cite{F.H.S} it is known that an artinian standard graded $\kk$-algebra $A=S/I$ is Gorenstein if and only if there exists $F\in R_d$ such that $I=\Ann_S(F)$.  
In other words, the inverse system of a graded artinian Gorenstein (AG) algebra is generated by only one form. We recall the definition of the Jordan degree type, \cite[Definition 2.28]{IMM}. 
\begin{definition}[Jordan degree type]
    Let $A$ be a graded artinian algebra and $\ell\in A_1$. Suppose that $P_{\ell,A}=(p_1,\dots ,p_t)$ is the Jordan type for $\ell$ and $A$, then there exist elements $z_1, \dots z_t\in A$, which depend on $\ell$, such that $\{\ell^iz_k\mid 1\leq k\leq t, 0\leq i\leq p_k-1\}$ is  a $\kk$-basis for $A$. The Jordan blocks of the multiplication map by $ \ell$ is determined by the strings $\mathsf{s}_k=\{z_k, \ell z_k, \dots , \ell^{p_k-1}z_k\}$, and $A$ is the direct sum $A=\langle \mathsf{s}_1\rangle \oplus\dots \oplus  \langle \mathsf{s}_t\rangle$. Denote by $d_k$ the degree of $z_k$. Then the \emph{Jordan degree type},  is defined to be the indexed partition $\mathcal{S}_{\ell,A}=({p_1}_{d_1}, \dots ,{p_t}_{d_t})$.
\end{definition}
We recall the definition of rank matrices for AG algebras and briefly describe its connection to Jordan degree types, see \cite{Al} for more details. We convent that the row and column indices of the matrices in this paper start from zero.
\begin{definition}
Let $A=S/\Ann(F)$ be an AG algebra with socle degree $d$. For alinear form $\ell\in A$ define the \textit{rank matrix}, $M_{\ell,A}$, of $A$ and $\ell$ to be  the upper triangular square matrix of size  $d+1$ with the following entries
$$(M_{\ell,A})_{i,j} = \rk\left(\times \ell^{j-i} : A_i\longrightarrow A_j\right),$$
for every $i\leq j$. For $i>j$ we set $ (M_{\ell,A})_{i,j}=0$.

For a given rank matrix $M_{\ell,A}$ we define
the \emph{Jordan degree type matrix}, $J_{\ell,A}$, of $A$ and $\ell$ to be the upper triangular matrix of size $d+1$ with the following non-negative entries 
\begin{align*}\label{J(A,l)definition}
(J_{\ell,A})_{i,j} :
=& (M_{\ell,A})_{i,j}+(M_{\ell,A})_{i-1,j+1}-(M_{\ell,A})_{i-1,j}-(M_{\ell,A})_{i,j+1},
\end{align*}
where we set $(M_{\ell,A})_{i,j}=0$ if either  $i< 0$ or $j< 0$.
\end{definition}
Matrix $J_{\ell,A}$  encodes the same information as the Jordan degree type $\mathcal{S}_{\ell,A}$ which is described in detail in \eqref{JDT}.

For each AG algebra $A=S/\Ann(F)$ with socle degree $d$ and linear form $\ell$ define new AG algebras $A^{(i)}_\ell:=S/\Ann(\ell^i\circ F)$ for all $0\leq i\leq d$. 

For every $0\le i\le d$ we define the $i$-th diagonal of $M_{\ell,A}$ to be the following vector
$$
\mathrm{diag}(i,M_{\ell,A}):=\big((M_{\ell,A})_{0,i},(M_{\ell,A})_{1,i+1},\dots , (M_{\ell,A})_{d-i,d}\big).
$$
By \cite[Proposition 3.4]{Al}, we know that $\mathrm{diag}(i,M_{\ell,A})=h_{A^{(i)}_\ell}$ for every $0\le i\le d$.

Consider the decreasing sequence  $$\mathbf{d}:=(\dim_{\kk}A^{(0)}_\ell, \dim_{\kk}A^{(1)}_\ell, \dots , \dim_{\kk}A^{(d)}_\ell),$$ and recall that the second difference sequence of $\mathbf{d}$ is denoted by $\Delta^2 \mathbf{d}$ and its  $i$-th entry is given by 
\begin{equation}\label{delta}
\Delta^2 \mathbf{d}(i)=\dim_{\kk}A^{(i)}_\ell +\dim_{\kk}A^{(i+2)}_\ell-2\dim_{\kk}A^{(i+1)}_\ell,
\end{equation}
where we set $\dim_{\kk}A^{(i)}_\ell=0$ for $i>d$. Then obtain the Jordan type of $A$ and $\ell$ from \cite[Proposition 3.14]{Al}.
\begin{equation}\label{P}
P_{\ell,A} = \big(\underbrace{d+1,\dots ,d+1}_{n_{d}} ,\underbrace{d,\dots ,d}_{n_{d-1}} ,\dots ,\underbrace{2,\dots ,2}_{n_1},\underbrace{1,\dots ,1}_{n_0}\big),
\end{equation}
such that $\mathbf{n}=(n_0,n_{1}, \dots , n_d)=\Delta^2 \mathbf{d}.$

One can see that $n_i=\sum_{k-j=i}(J_{\ell,A})_{j,k}$ for each $0\le i\le d$. Moreover, to obtain the Jordan degree type, one can read off the degrees of the entries in $P_{\ell,A}$ from the row indices of non-zero elements in $J_{\ell,A}$. More precisely, 
\begin{align}\label{JDT}
    \mathcal{S}_{\ell,A} =& \big(\underbrace{{d+1}_0,\dots ,{d+1}_0}_{n_{d,0}} ,\underbrace{{d}_0,\dots ,d_0}_{n_{d-1,0}}, \underbrace{{d}_1,\dots ,d_1}_{n_{d-1,1}},\\\nonumber &\underbrace{{d-1}_0,\dots ,{d-1}_0}_{n_{d-2,0}},\dots ,\underbrace{{d-1}_2,\dots ,{d-1}_2}_{n_{d-2,2}},\\\nonumber
    &\hspace{4cm}\vdots\\\nonumber
    &\underbrace{2_0,\dots ,2_0}_{n_{1,0}},\dots ,\underbrace{2_d,\dots ,2_d}_{n_{1,d-1}},\underbrace{1_0,\dots ,1_0}_{n_{0,0}},\dots ,\underbrace{1_{d+1},\dots ,1_{d+1}}_{n_{0,d}}\big),
\end{align}
where 
$n_{i,j}=(J_{\ell,A})_{j,i+j}$ and $n_i=\sum_{j=0}^{d-i}n_{i,j}$.
\begin{example}
Let $A=S/\Ann(F)=\kk[x,y,z]/\Ann(F)$ such that $F = X^4+XY^2Z$. Let $\ell=y$  and observe that 
\begin{align*}
    & A^{(0)}_y=A , \quad \text{} \hspace*{2mm} h_{A}=(1,3,5,3,1),\\
    &A^{(1)}_y=S/\Ann(2XYZ), \quad \text{} \hspace*{2mm} h_{A^{(1)}_y}=(1,3,3,1),\\
    &A^{(2)}_y=S/\Ann(2XZ), \quad \text{} \hspace*{2mm} h_{A^{(1)}_y}=(1,2,1), \hspace*{1mm}\text{and}\\
    &h_{A^{(3)}_y}=h_{A^{(4)}_y}=(0).\\
\end{align*}
So we have that 
\begin{align*}
    M_{y,A}=\begin{bmatrix}
        1&1&1&0&0\\
        0&3&3&2&0\\
        0&0&5&3&1\\
        0&0&0&3&1\\
        0&0&0&0&1\\
    \end{bmatrix}, \hspace{1mm}\text{and}\hspace{2mm}J_{y,A}=\begin{bmatrix}
        0&0&1&0&0\\
        0&0&0&2&0\\
        0&0&1&0&1\\
        0&0&0&0&0\\
        0&0&0&0&0\\
    \end{bmatrix}.
\end{align*}
Therefore, using equations \eqref{P} and \eqref{JDT} we  obtain that $$P_{y,A}=(3,3,3,3,1),\hspace{1mm}\text{and}\hspace{2mm}\mathcal{S}_{y,A}=(3_0,3_1,3_1,3_2,1_2).$$
\end{example}
\begin{remark}
   The rank matrices were further studied in \cite{Al} and three necessary conditions were provided for an upper triangular square matrix with non-negative entries that can occur as the rank matrix of some AG algebra and a linear form.
\end{remark}


\section{Artinian Gorenstein algebras of codimension two }
In this section, we show that for AG algebras of codimension two, which are complete intersections, the Jordan type uniquely determines the Jordan degree type. The complete characterization of possible Jordan types of such algebras is provided in \cite{CIJT}. 
\begin{theorem}\label{theorem}
    Let $P=(P_1,\dots , P_t)$ be the Jordan type of some AG algebra of codimension two, socle degree $d$, and a linear form. Then the Jordan degree type associated to $P$ is uniquely determined.
\end{theorem}
\begin{proof}
Note that $P$ is of the form in equation \eqref{P}. Using equation \eqref{delta}, we have that $\dim_\kk A^{(d)}_\ell=n_d$, $\dim_\kk A^{(d-1)}_\ell=n_{d-1}-n_d$, and recursively we obtain $\dim_\kk A^{(i)}_\ell$, for every $0\le i\le d-2$ as the following 
$$\dim_\kk A^{(i)}_\ell=n_i+2\dim_\kk A^{(i+1)}_\ell-\dim_\kk A^{(i+2)}_\ell.$$
We claim that for each $0\le i\le d$ the $i$-th the diagonal of the rank matrix $M_{\ell,A}$, $\mathrm{diag}(i,M_{\ell,A})$, that is equal to the Hilbert function of $A^{(i)}_\ell$ is  uniquely determined by $\dim_\kk A^{(i)}_\ell$ and the socle degree of $A^{(i)}_{\ell}$ that is equal to $d-i$. To show the claim, let $B$ be any AG algebra with socle degree $s$ and $\dim_\kk B=t$. We obviously have $t\geq s+1$ and since $h_B$ increases by at most one in each degree we have that $t = \sum_{i=0}^{r}(s+1-2i)$ where $r$ is the largest degree $i$ where $h_B(i)=i+1$. Therefore,
$$
h_B(i)= \begin{cases}
i+1,& 0\le i\le r,\\
r+1, & r+1\le i\le \lfloor\frac{s}{2}\rfloor,
\end{cases}
$$
and by symmetry we obtain the Hilbert function of $B$.
Thus, we conclude that using equation \eqref{JDT} the Jordan degree type matrix $J_{\ell,A}$ is uniquely determined and so is the Jordan degree type $\mathcal{S}_{\ell,A}$.
\end{proof}


\section{Artinian Gorenstein algebras of higher codimensions}
In this section we show that the Jordan degree type is a finer invariant than the Jordan type for AG algebras of codimension at least three.  Before that, in the following example we describe an algorithm that finds an AG algebra of codimension three and a linear form such that the associated rank matrix is equal to a given matrix satisfying the necessary conditions provided in \cite[Corollary 3.9]{Al}.  \\
We first briefly recall the \emph{catalecticant matrix} associated to $F\in R_d$ that is the Macaulay dual generator of $A=S/\Ann(F)$. The entries of the $j$-th catalecticant matrix of $F$ with respect to $\kk$-bases $\mathcal{B}_j=\{\alpha_u^{(j)}\}_u$ and  $\mathcal{B}_{d-j}=\{\beta_v^{(d-j)}\}_v$ of $S_j$ and $S_{d-j}$ respectively is  defined to be 
$$
\Cat_j(F)=\big(\alpha_u^{(j)}\beta_v^{(d-j)} \circ F\big).
$$
For every $j$, the rank of $\Cat_j(F)$ is equal to $h_A(j)$, the Hilbert function of $A$ in degree $j$, see \cite{IK}.

\begin{example}[Constructing $A$ and $\ell\in A_1$ having a given rank matrix]\label{ex1}
    Consider $$M=\begin{bmatrix}
1&1&1&1&0&0\\
0&3&2&2&1&0\\
0&0&4&3&2&1\\
0&0&0&4&2&1\\
0&0&0&0&3&1\\
0&0&0&0&0&1\\
\end{bmatrix}$$ and note that it satisfies the necessary conditions in \cite[Corollary 3.9]{Al}. So we determine whether $M$ occurs as the rank matrix of an AG algebra $A$ and a linear form $\ell$. By linear change of coordinates we may assume that $\ell=x$. Suppose $M=M_{x,A}$ for $A=S/\Ann(F)=\kk[x,y,z]/\Ann(F)$. The form $F$ can be written as
$$
F = X^5G_0+X^4G_1+\cdots +G_5,
$$
such that $G_i$ are forms of degree $i$ in $\kk[Y,Z]$.  We will find $G_i$'s recursively. The last two diagonals of $M$ are equal to $h_{A^{(5)}_x}$ and $h_{A^{(4)}_x}$ and since they are equal to zero vectors we get that $G_0=G_1=0$. The dual generator of $A^{(3)}_x$ is $x^3\circ F = G_2$. We have $h_{A^{(3)}_x}=(1,1,1)$ so $G_2$ is the second power of a linear form in $\kk[Y,Z]$, thus it can be chosen to be equal to $G_2=Y^2$. Therefore, 
$$
F = X^3Y^2+X^2G_3+XG_4+G_5.
$$
Continuing this process we get that $h_{A^{(2)}_x}=(1,2,2,1)$ where the dual generator of $A^{(2)}_x$ is equal to $x^2\circ F=6XY^2+2G_3$. We note that for $G_3=0$ we get the desired Hilbert function for $A^{(2)}_x$. To obtain $G_4$  we compute $x\circ F =3X^2Y^2+G_4$ and to get $h_{A^{(1)}_x}=(1,2,3,2,1)$ we choose $G_4=0$. Thus we have 
$$
F=X^3Y^2+G_5
$$
and it remains to find for a polynomial $G_5$ for which we get $h_A=(1,3,4,4,3,1)$. We do so by computing the catalecticant matrices of $F$, we recall that the Hilbert function of $A$ in degree $1$ and $2$ are equal to ranks of the first and second catalecticant matrices respectively. We denote $G_5=\sum_{i=0}^5a_iY^{5-i}Z^{i}$. For each $j$ let $\mathcal{B}_j$ be the monomial basis of $S_j$ in the lexicographic order. Then first and second catalecticant matrices of $F$ are block matrices. 
\begin{align*}
\mathrm{Cat}_1(F)&=
\left[\begin{array}{@{}c|cc|ccc|cccc|ccccc@{}}
0& 0&0&12&0&0 &0&0&0&0&0&0&0&0&0\\\hline
0& 12&0&0&0&0& 0 &0&0&0&a_0&a_1&a_2&a_3&a_4\\
0& 0&0&0&0&0 &0&0&0&0&a_1&a_2&a_3&a_4&a_5
\end{array}
\right],\\
\end{align*}
and
\begin{align*}
\mathrm{Cat}_2(F)&=
\left[\begin{array}{@{}c|cc|ccc|cccc@{}}
0& 0&0&12&0&0&0&0&0&0 \\\hline
0& 12&0&0&0&0&0&0&0&0\\
0& 0&0&0&0&0&0&0&0&0 \\\hline
12& 0&0&0&0&0&a_0&a_1&a_2&a_3 \\
0& 0&0&0&0&0&a_1&a_2&a_3&a_4\\
0& 0&0&0&0&0&a_2&a_3&a_4&a_5
\end{array}
\right].\\
\end{align*}
One can easily see that for $G_5=(Y+Z)^5$  we have $\rk \mathrm{Cat}_1(F)=3$ and $\rk \mathrm{Cat}_2(F)=4$. We conclude that $M$ is the rank matrix of the AG algebra $A=S/\Ann(F)$ and $\ell=x$ where $$F = X^3Y^2+(Y+Z)^5.$$
Furthermore, using equations \eqref{P} and \eqref{JDT} the Jordan type and Jordan degree type of $A$ and $\ell=x$ are equal to 
$P_{x,A}=(4,4,4,1,1,1,1)$ and $\mathcal{S}_{x,A}=(4_0,4_1,4_2,1_1,1_2,1_3,1_4)$.
\end{example}

Using the above method we show that Jordan degree type encodes more information than the Jordan type for AG algebras of codimension at least three.


\begin{example}[Two Jordan degree types for a Jordan type]\label{ex}
The following matrices satisfy the necessary conditions given in \cite[Corollary 3.9]{Al}. So they can occur as rank matrices for AG algebras and linear forms. 
\begin{Small}\begin{equation*}
M_1 = \begin{bmatrix}
1&\textcolor{red!70}{1}&1&1&1&1&1&0&0\\
0&3&\textcolor{red!70}{3}&3&3&2&2&1&0\\
0&0&6&\textcolor{red!70}{6}&5&4&3&2&1\\
0&0&0&9&\textcolor{red!70}{6}&5&4&2&1\\
0&0&0&0&9&\textcolor{red!70}{6}&5&3&1\\
0&0&0&0&0&9&\textcolor{red!70}{6}&3&1\\
0&0&0&0&0&0&6&\textcolor{red!70}{3}&1\\
0&0&0&0&0&0&0&3&\textcolor{red!70}{1}\\
0&0&0&0&0&0&0&0&1\\
\end{bmatrix}, M_2 = \begin{bmatrix}
1&\textcolor{red!70}{1}&1&1&1&1&1&0&0\\
0&3&\textcolor{red!70}{3}&3&3&2&2&1&0\\
0&0&6&\textcolor{red!70}{5}&5&4&3&2&1\\
0&0&0&9&\textcolor{red!70}{7}&5&4&2&1\\
0&0&0&0&9&\textcolor{red!70}{7}&5&3&1\\
0&0&0&0&0&9&\textcolor{red!70}{5}&3&1\\
0&0&0&0&0&0&6&\textcolor{red!70}{3}&1\\
0&0&0&0&0&0&0&3&\textcolor{red!70}{1}\\
0&0&0&0&0&0&0&0&1\\
\end{bmatrix}.
\end{equation*}\end{Small}
Note that the two matrices only differ in the red entries.  Similar to the process described in Example \ref{ex1} but with more technicality, we find that both $M_1$ and $M_2$ occur as rank matrices. More precisely, 
$M_1=M_{x,A}$ and $M_2=M_{x,B}$ such that $A=S/\Ann(F)$, $B=S/\Ann(G)$, 
$$F=X^6Y^2+X^3(Y+Z)^5+XY(Y+Z)^6+Y^8+Z^8,$$ and 
$$G=X^6Y^2+X^3(Y^5+Z^5)+X(Y^7+Z^7)+Y^8+Y^7Z+YZ^7+Z^8.$$
Because the sum of red entries of the two matrices are equal using equation \eqref{P} we get that $$P_{x,A}=P_{x,B}=(7^3,4^4,2^2,1^6).$$ On the other hand, by equation \eqref{JDT} we have 
\begin{align*}
\mathcal{S}_{x,A}&=(7_0,7_1,7_2,4_1,4_2,4_3,4_4,2_2,2_5,1_3,1_3,1_4,1_4,1_5,1_5),\\
\mathcal{S}_{x,B}&=(7_0,7_1,7_2,4_1,4_2,4_3,4_4,2_3,2_4,1_2,1_3,1_3,1_5,1_5,1_6).
\end{align*}

\end{example}
\begin{remark}
Using the algebra software Macaulay2, it has been computationally proven that among all AG algebras of codimension three with socle degree smaller than $8$ there are no examples of AG algebras with the same Jordan type and different Jordan degree types.  In other words, the above example has the smallest socle degree among all AG algebras of codimension three. On the other hand, for AG algebras of higher codimensions there exist such examples of smaller socle degrees, Example \ref{codim4Ex}. 
\end{remark}
Using the same method as above we obtain the following example which we leave the computations to the reader. One may use the Macaulay2 codes provided bellow for further experiments. 
\begin{example}\label{codim4Ex}
    Let $S = \kk[x,y,z,w]$ and $A=S/\Ann(F)$ and $B=S/\Ann(G)$ be AG algebras of codimension four where 
$$F = X^4Y+X^2Y^2Z+XY^3W +Y^3W^2,$$ and $$G = X^4Y+X^2Y^2Z+XZ^3Y+W^5. 
$$
We have that $h_A=h_B=(1,4,7,7,4,1)$. For linear form $\ell=x$ the rank matrices of $A$ and $B$ are equal to
\begin{equation*}
M_{\ell,A} = \begin{bmatrix}
1&\textcolor{red!70}{1}&1&1&1&0\\
0&4&\textcolor{red!70}{4}&3&2&1\\
0&0&7&\textcolor{red!70}{4}&3&1\\
0&0&0&7&\textcolor{red!70}{4}&1\\
0&0&0&0&4&\textcolor{red!70}{1}\\
0&0&0&0&0&1\\
\end{bmatrix}, \hspace{2mm}\text{and}\hspace{2mm} M_{\ell,B} = \begin{bmatrix}
1&\textcolor{red!70}{1}&1&1&1&0\\
0&4&\textcolor{red!70}{3}&3&2&1\\
0&0&7&\textcolor{red!70}{6}&3&1\\
0&0&0&7&\textcolor{red!70}{3}&1\\
0&0&0&0&4&\textcolor{red!70}{1}\\
0&0&0&0&0&1\\
\end{bmatrix}.
\end{equation*}
Using equation \eqref{P} and \eqref{JDT} we get that 
$$
P_{\ell,A}=P_{\ell,B}=(5,5,3,3,2,2,1,1,1,1),
$$
and
\begin{align*}
    \mathcal{S}_{\ell,A}&=(5_0,5_1,3_1,3_2,2_1,2_3,1_2,1_2,1_3,1_3),\\
    \mathcal{S}_{\ell,B}&=(5_0,5_1,3_1,3_2,2_2,2_2,1_1,1_2,1_3,1_4).\\
\end{align*}

\end{example}
\noindent \textbf{Macaulay 2 codes:} We provide Macaulay2 codes that were primarily used for computations in this study. Let $F$ be a form of degree $d$ in the Macaulay dual ring, $R=\kk[x,y,z]$ (here, we abuse the notation and denote the variables in the dual ring by $x,y,z$).  The following are some codes that compute the catalecticant matrix of $F$; using this, we compute the Hilbert function of the AG algebra $A$ associated with $F$ without computing $A$. We then for a given linear form $\ell$ we compute the dual generators of the AG algebras $A^{(i)}_\ell$ and then we compute their Hilbert functions which form the diagonals of the rank matrix $M_{\ell,A}$, $\mathrm{diag}(i,M_{\ell,A})$.  Note that the following computations work for any number of variables, one only needs to adjust the number of variables in $R$ and $\ell$ and everything else remains the same. 

\begin{verbatim}
R = QQ[x,y,z]
Cat = (i,F) -> (diff(transpose basis(i,R), 
                diff(basis(first degree F-i,R),F)))
Hilb = (F) -> for i to first degree F list rank Cat(i,F)
l = a*x+b*y+c*z;
f = (i,l,F) -> diff(l^i,F)
diag = (i,l,F) -> Hilb(diff(l^i,F))
\end{verbatim}
Finally, given $F$ together with its degree, the following two functions compute the rank matrix and Jordan type of AG algebra associated with $F$ and $\ell$: 
\begin{verbatim}
M = (l,F) -> ( 
    d = first degree F;  
    L = for i to d list for j to d-i list 
        rank Cat(i,diff(l^j,F));
    matrix({L_0}|(for i from 1 to d list 
          (for j from 1 to i list 0)|L_i))
    )
    
P = (l,F) -> ( 
    d = first degree F; 
    L1 = for i to d list sum for j to d-i list 
         rank Cat(j,diff(l^i,F));
    L2 = L1|{0,0};
    L3 = for i to d list L2_i+L2_(i+2)-2*L2_(i+1);
         reverse flatten for i to d list if L3_i!= 0 then 
         for j to L3_i-1 list i+1 else continue
    )

\end{verbatim}
We end by posing some problems: 
\begin{problem}
Under what conditions for an AG algebra of codimension higher than two, is the Jordan degree type equivalent to the Jordan type?
\end{problem}
\begin{problem}
For a given Jordan type, can we determine the possible Jordan degree types?
\end{problem}
\renewcommand*{\bibfont}{\normalfont\footnotesize}

\printbibliography
\end{document}